\documentclass[12 pt]{article}
\usepackage{fullpage}

\usepackage{lipsum}
\usepackage{amsfonts}
\usepackage{graphicx}
\usepackage{epstopdf}
\usepackage{standalone}
\usepackage{algcompatible}

\ifpdf
  \DeclareGraphicsExtensions{.eps,.pdf,.png,.jpg}
\else
  \DeclareGraphicsExtensions{.eps}
\fi

\usepackage{amsopn}

\usepackage{epsfig} %
\usepackage{amsmath} %
\usepackage{amssymb}  %
\usepackage{bm}
\DeclareMathAlphabet{\mathcal}{OMS}{cmsy}{m}{n}
\usepackage[english]{babel}
\usepackage{url}
\usepackage{algorithm}
\usepackage{color}
\usepackage{pgfplots}
\usepackage{siunitx}
\usepackage{tikz}
\usepackage{tkz-euclide}
\usepackage{chngcntr}
\usepackage{verbatim}
\usepackage{enumerate}

\definecolor{ao(english)}{rgb}{0.0, 0.5, 0.0}
\usepackage[colorlinks, citecolor = {ao(english)}, linkcolor = {ao(english)}]{hyperref} 
\usepackage{cleveref}
\usepackage{aliascnt}
\usetikzlibrary{calc}
\usepackage{subcaption}
\usepackage{tikz}
\usepackage{pgfplots}
\usetikzlibrary{arrows,shapes,trees,calc,positioning,patterns,decorations.pathmorphing,decorations.markings}
\usetikzlibrary{matrix}
\usepgfplotslibrary{groupplots}
\pgfplotsset{compat=newest}
\usepackage{amsthm}
\usepgfplotslibrary{patchplots}

\crefname{figure}{Fig.}{Fig.}

\newtheorem{thm}{Theorem}
\crefname{thm}{Theorem}{Theorems}

\newtheorem{prop}{Proposition}
\crefname{prop}{Proposition}{Propositions}
\newtheorem{lem}{Lemma}
\crefname{lem}{Lemma}{Lemmas}
\newtheorem{cor}{Corollary}
\crefname{cor}{Corollary}{Corollaries}
\theoremstyle{remark}
\newtheorem{rem}{Remark}
\crefname{rem}{Remark}{Remarks}

\theoremstyle{definition}
\newtheorem{example}{Example}
\crefname{example}{Example}{Examples}

\crefname{ass}{Assumption}{Assumption}
\usepackage{dsfont}
\let\mathbb=\mathds
\usepackage[numbers,sort&compress]{natbib}

\crefname{conj}{Conjecture}{Conjectures}

\theoremstyle{definition}
\newtheorem{defn}{Definition}
\crefname{defn}{Definition}{Definitions}

\crefname{prob}{Problem}{Problems}
\crefname{algorithm}{Algorithm}{Algorithms}

\newcommand{\Rmn}{\mathbb{R}^{n \times m}}

\newcommand{\Rnn}{\mathbb{R}^{n \times n}}

\newcommand{\rk}{\textnormal{rank}}

\newcommand{\diag}{\textnormal{diag}}

\newcommand{\transp}{\mathsf{T}}

\newcommand{\tp}[1]{{#1\text{-positive}}}

\newcommand{\Hank}[1]{\mathcal{H}_{#1}}

\newcommand{\lint}{\ell_1}
\newcommand{\linf}{\ell_\infty}

\newcommand{\compound}[2]{#1_{[#2]}}
\newcommand{\Con}[1]{{\mathcal{C}^{#1}}}
\newcommand{\Obs}[1]{{\mathcal{O}^{#1}}}

\colorlet{FigColor1}{blue}
\colorlet{FigColor2}{red}
\colorlet{FigColor3}{ao(english)}
\colorlet{FigColor4}{orange}
\pgfplotsset{every axis plot/.append style={line width=1.5pt}}

	\definecolor{bluebell}{rgb}{0.64, 0.64, 0.82}
		\definecolor{airforceblue}{rgb}{0.36, 0.54, 0.66}
	
\crefformat{equation}{\textup{#2(#1)#3}}
\crefrangeformat{equation}{\textup{#3(#1)#4--#5(#2)#6}}
\crefmultiformat{equation}{\textup{#2(#1)#3}}{ and \textup{#2(#1)#3}}
{, \textup{#2(#1)#3}}{, and \textup{#2(#1)#3}}
\crefrangemultiformat{equation}{\textup{#3(#1)#4--#5(#2)#6}}%
{ and \textup{#3(#1)#4--#5(#2)#6}}{, \textup{#3(#1)#4--#5(#2)#6}}{, and \textup{#3(#1)#4--#5(#2)#6}}

\Crefformat{equation}{#2Equation~\textup{(#1)}#3}
\Crefrangeformat{equation}{Equations~\textup{#3(#1)#4--#5(#2)#6}}
\Crefmultiformat{equation}{Equations~\textup{#2(#1)#3}}{ and \textup{#2(#1)#3}}
{, \textup{#2(#1)#3}}{, and \textup{#2(#1)#3}}
\Crefrangemultiformat{equation}{Equations~\textup{#3(#1)#4--#5(#2)#6}}%
{ and \textup{#3(#1)#4--#5(#2)#6}}{, \textup{#3(#1)#4--#5(#2)#6}}{, and \textup{#3(#1)#4--#5(#2)#6}}

\crefdefaultlabelformat{#2\textup{#1}#3}

\begin{document}

\title{Balanced truncation of $k$-positive systems
}
\author{Christian Grussler, Tobias Damm and Rodolphe Sepulchre 
	\thanks{The first author is with the Department of Electrical
          Engineering and Computer Sciences at UC Berkeley,
          Berkeley, CA 94720, USA, {\tt\small
            christian.grussler@berkeley.edu}. The second author is
          with the Department of Mathematics at TU Kaisers\-lautern,
          Gottlieb-Daimler-Stra\ss{}e 48, and Fraunhofer
ITWM, 67663 Kaiserslautern, Germany, {\tt\small  damm@mathematik.uni-kl.de}. The third author is with the Control Group at the Department of Engineering, University of Cambridge, Trumpington Street, Cambridge CB2 1PZ, United Kingdom, {\tt\small  r.sepulchre @eng.cam.ac.uk}.
		}
	}

\maketitle

	\begin{abstract}
This paper considers balanced truncation of discrete-time Hankel $k$-positive systems, characterized by Hankel matrices whose minors up to order $k$ are nonnegative. Our main result shows that if the truncated system has order $k$ or less, then it is Hankel totally positive ($\infty$-positive), meaning that it is a sum of first order lags. This result can be understood as a bridge between two known results: the property that the first-order truncation of a positive system is positive ($k=1$), and the property that   balanced truncation preserves state-space symmetry. It provides a broad class of systems where balanced truncation is guaranteed to result in a minimal internally positive system. 	
	
\end{abstract}

\section{Introduction}
Model order reduction aims at facilitating analysis, design, and implementation of systems by finding simpler lower order approximations. Standard techniques such as balanced truncation provide qualitatively good approximations in reproducing the input-output behaviour. {{But it is widely unclear in which cases these approximations can be realized through the parallel interconnection of first order lags only.}} Such approximations, also known as relaxation systems \cite{willems1976realization}, have been of considerable interest as they are passive, externally (input-output) positive and as recently shown often admit sparse, scalable
        optimal controllers \cite{pates2019optimal}. While balanced truncation and optimal Hankel norm
        approximation are known to preserve this property (in continuous-time) for any
        order \cite{liu1998model}, it is an open question in which cases this
        property can be gained from non-relaxation systems. Here, we
        provide a first answer by showing that balanced truncation of
        so-called \emph{Hankel $k$-positive single-input-single-output
          (SISO) systems} yields such approximations, if the reduced
        order is no larger than $k$.  
        
	{In discrete-time, Hankel $k$-positive systems are defined as systems whose Hankel operator has a $k$-positive matrix representation, i.e., all its minors of order up to $k$ are nonnegative.} %
    For example, Hankel $\tp{1}$ systems correspond to the well-known (strictly proper) externally (input-output) positive systems \cite{Farina2000}. %
        
    As recently discovered in \cite{grussler2020variation}, {under a mild multiplicity assumption, Hankel $k$-positive systems are dominated by relaxation systems of order $k$, i.e., after a partial fraction decomposition, the sum corresponding to the $k$ largest poles in magnitude has a relaxation system structure. This forms a bridge between externally positive (one dominant first order lag) and relaxation systems (sum of first order lags).} By our main result, balanced truncation preserves the structure of the  dominant parts of the system. Specifically,  external positivity is preserved by truncation of SISO systems to order~$1$. 
               
Many externally positive systems are modelled by internally positive realizations, i.e., system matrices with nonnegative entries. This property 
is appealing in scalable stability analysis
\cite{farina2011positive,rantzer2015scalable,son1996robust,tanaka2011bounded} and
enjoyed by many compartmental models, e.g., within bio-chemistry,
economics, or transportation,
\cite{farina2011positive,luenberger1979introduction}. Several
methods have been suggested to preserve internal positivity in the reduction process,
\cite{reis2009positivity,feng2010internal,sootla2012scalable}.
Unfortunately, even for relaxation systems these methods {often yield
conservative results, which can be outperformed by balanced truncation
to much lower orders} (see~\cref{sec:example} and \cite{grussler2012symmetry,grussler2012model} for examples). In contrast, %
for the class of Hankel $k$-positive SISO systems, we show that balanced truncation preserves
internal positivity. In the multi-input-multi-output (MIMO)
case, our results remain valid as long as the Hankel operator is
symmetric and its representation matrix is $k$-positive. %
We believe that the framework of
$k$-positivity also provides a natural extension beyond that. The fact that balanced truncation to
order~$1$ preserves internal positivity also for MIMO systems, %
\cite{grussler2012symmetry}, is an indication.

	The paper is organized as follows. In the preliminaries, we review discrete-time systems and the relationship between Kung's algorithm and balanced truncation, which is essential in the proof of our main result. Then, we summarize the relevant parts of the $k$-positivity theory from \cite{grussler2020variation}. \Cref{sec:Main} contains our main results on the truncation of Hankel $k$-positive systems. Finally, we discuss extensions to MIMO systems and conclude with an illustrative example.  

\section{Preliminaries}
	\label{sec:prelim}
	\subsection{Notation}
	\subsubsection{Sets}
	In this work, the set of nonnegative reals and integer are denoted by $\mathds{R}_{\geq 0} = [0,\infty)$ and $\mathds{Z}_{\geq 0} = \mathds{N}_{0}$, respectively. Further, for $k, l \in \mathds{Z}$, we use $(k:l) := \{k,k+1,\dots,l\}$, $k \leq l$.
	\subsubsection{Matrices}
	For real valued matrices $X = (x_{ij}) \in \Rmn$, including
        vectors $x = (x_i) \in \mathbb{R}^n$, we say that $X$ is
        \emph{nonnegative}, $X\ge 0$ or $X \in \Rmn_{\geq 0}$, if
        all elements $x_{ij} \in \mathbb{R}_{\geq 0}$; we use the
        corresponding notation for positive matrices. If $X
        \in \Rnn$, then $\sigma(X) =
        \{\lambda_1(X),\dots,\lambda_n(X)\}$ denotes its
        \emph{spectrum}, where the eigenvalues are ordered by
        descending absolute value, i.e., $\lambda_1(X)$ is the
        eigenvalue with the largest magnitude, counting multiplicity.
        In case that the magnitude of two eigenvalues coincides, we
        sub-sort them by decreasing real part. A matrix $X$ is called \emph{reducible}, if there exists a permutation matrix $P=\begin{pmatrix}
        P_1 & P_2
        \end{pmatrix}$ so that $P_2^TXP_1=0$; otherwise $X$ is
        \emph{irreducible}. We call $X$ \emph{Hankel}, if it is constant along its anti-diagonals. Further, $X$ is \emph{positive semidefinite}, $X \succeq 0$, if $X = X^\transp$ and $\sigma(X) \subset \mathbb{R}_{\geq 0}$. The identity matrix in $\Rnn$ is denoted by $I_n$ and the Moore-Penrose \emph{pseudo-inverse} of $X \in \Rmn$ by $X^\dagger$.
	Finally, a \emph{(consecutive) $j$-minor} of $X$ in $\Rmn$ is
        defined as a minor which is constructed of (consecutive) $j$
        columns and $j$ rows of $X$. The submatrix with rows $I \subset
        [1:n]$ and columns $J \subset [1:m]$ is written as $X_{\{I,J\}}$. 

	\subsubsection{Functions}
We consider functions $g: \mathds{Z} \to \mathbb{R} \cup \{\pm \infty\}$. \emph{Nonnegative functions} $g: \mathds{Z} \to \mathbb{R}_{\geq0}$ are written as $g \geq 0$ and snapshots as $g(i:j) := \begin{pmatrix}
g(i) & \dots & g(j)
\end{pmatrix}^\transp$. The \emph{(1-0) indicator function} of $\mathcal{S} \subset \mathds{Z}$ is defined as
	\begin{align*}
	\mathds{1}_{\mathcal{S}}(t) := \begin{cases}
	1 & t \in \mathcal{S}\\
	0 & t \notin \mathcal{S}
	\end{cases}
	\end{align*}
	which then defines the \emph{unit impulse function} as $\delta(t)
        := \mathds{1}_{\{0\}}(t)$.
The set of all \emph{absolutely summable functions} is denoted by $\lint$ and the set of \emph{bounded functions} by $\linf$.

\subsection{Linear discrete-time systems}
We consider \emph{linear discrete-time time-invariant systems} %
\begin{align*}
  x(t+1)&=Ax(t)+bu(t),\quad 
y(t)=cx(t)
\end{align*}
with $A\in\mathbb{R}^{n\times n}$, $b,c^\top\in\mathbb{R}^n$.
The output $y(t)=g(t) = cA^{t-1}b$ corresponding to initial state $x(0)=0$ and
input $u = \delta$ is called the \emph{impulse response}.
The transfer function is given by $G(z) = c(zI_n-A)^{-1}b$. It can be
written as
\begin{equation}\label{eq:defGz}
G(z) = \sum_{t=0}^\infty g(t)z^{-t} = \frac{r \prod_{i=1}^{m} (z-z_i)}{\prod_{j=1}^{n}(z-p_i)}\;,
\end{equation}
where $m < n$, $r \in \mathbb{R}$, $p_i$ and $z_i$ are referred to as
\emph{poles} and \emph{zeros}, which are both sorted in the same way
as the eigenvalues of a matrix. The triple $(A,b,c)$ is also called a
\emph{realization of $G$}. We always assume that
$\{z_1,\ldots,z_m\}\cap\{p_1,\ldots,p_n\}=\emptyset$, in which case
the realization is called \emph{minimal}. 
We also assume asymptotic stability, i.e.,
$|p_1|,\ldots,|p_n|<1$. Then, for $u\in \linf$ with $u(t) = u(t)s(t-1)$ and $t
\geq 0$, the Hankel operator associated to the
system is defined by
	\begin{align}
	(\Hank{g} u)(t) &:= \sum_{\tau=-\infty}^{-1} g(t-\tau)u(\tau)
                          = \sum_{\tau=1}^{\infty} g(t+\tau)u(-\tau).
                         \label{eq:def_hank_disc}  
	\end{align}
If we set $x_0=\sum_{\tau=-\infty}^{-1} A^{\tau+1}bu(\tau)$, then
$(\Hank{g}u)(t)$ equals the impulse response to $(A,x_0,c)$. The
operator is the limit (for $j\to\infty$) of the finite truncated matrix
representations $\Hank{g}^j u = H_g(1,j)u(-1:-j)$, 
where \small \begin{align*}%
H_g(t,j) &:= \begin{pmatrix}
g(t) & g(t+1) & \dots & g(t+j-1)\\
g(t+1) & g(t+2) & \dots & g(t+j)\\
\vdots & \vdots & \ddots   & \vdots \\
g(t+j-1)   & g(t+j) & \dots & g(t+2j-2)\\
\end{pmatrix}.
\end{align*}
\normalsize

\subsection{Balanced truncation}\label{sec:bal}
Given a minimal system realization $(A,b,c)$ of $G(z)$, let
\begin{subequations}
	\begin{align}
	\Con{N}(A,b) &:= \begin{pmatrix}
	b & Ab & \dots & A^{N-1}b
	\end{pmatrix} \\
	\Obs{N}(A,c) &:= \begin{pmatrix}
	c^\transp & A^\transp  c^\transp & \dots & {A^\transp}^{N-1} c^\transp
	\end{pmatrix}^\transp
	\end{align}
\end{subequations}
denote the \emph{finite-time controllability} and
\emph{observability} operators.  %
Accordingly, we define the  \emph{(finite-time) controllability,
  observability} and  \emph{cross-Gramian}
by  \begin{subequations}\label{eq:convergence_of_Gramians} 
\begin{align}
P(N) & := \Con{N}(A,b) \Con{N}(A,b)^\transp, \ P = \lim_{N \to \infty} P(N),\\
Q(N) &:=  \Obs{N}(A,c)^\transp \Obs{N}(A,c), \ Q = \lim_{N \to \infty} Q(N),\\
X(N) &:=   \Con{N}(A,b) \Obs{N}(A,c), \ X = \lim_{N \to \infty} X(N),
\end{align}
\end{subequations}
respectively. %
We call $(A,b,c)$ a \emph{finite-time balanced realization} if $P(N)=Q(N)$ is diagonal
with decreasing diagonal entries, called the \emph{finite-time Hankel singular values}. Note that with 
\begin{equation}
    H_g(1,N) := \Obs{N}(A,c)\Con{N}(A,b), \label{eq:hankel_fac}
\end{equation}
it holds that 
\begin{align*}
X(N)^2 &= \Con{N}(A,b) H_g(1,N) \Obs{N}(A,c) \\ &= \Con{N}(A,b) H_g(1,N)^\transp \Obs{N}(A,c) = P(N) {Q}(N).
\end{align*}
Therefore, \begin{equation}
\lambda_i (H_g(1,N)) = \lambda_{i}(X(N)), \ 1 \leq i \leq n \label{eq:EW_cross_finite}
\end{equation}
and if $(A,b,c)$ is finite-time balanced then $X(N)$
is diagonal. An analogous terminology is used in the limit case where we  drop the finite-time prefix and replace $H_g(1,N)$ by $\mathcal{H}_g$.

There always exists a (finite-time) balanced realization $(A,b,c)$ of $G(z)$, and a (finite-time) balanced truncated system approximation $G_r(z)$ of order $r$ is then given by the
 realization $(A_{(1:r),(1:r)},b_{(1:r))},c_{(1:r)})$. %
\subsection{Kung's algorithm}\label{sec:Kung}
Note that $H_g(2,N)=\Obs{N}(A,c)A\Con{N}(A,b)$ and for a minimal
realization, we have $$\rk\, H_g(1,N)=\rk\, \Obs{N}(A,c)=\rk\,{\cal
  C}(N)=\min\{n,N\}.$$ Assume $N\ge n$.  
If $H_g(1,N)=LR$ is a rank-revealing factorization, then the image of $\Obs{N}(A,c)$ equals
the image of $L$, i.e.\ $\Obs{N}(A,c)S=L$ for some nonsingular matrix
$S$, and $S^{-1}\Con{N}(A,b)=R$. We set  $\tilde c=L_{(1,:)}=cS$
and $\tilde{b}=R_{(:,1)}=S^{-1}b$.
The matrices  $\Obs{N}(A,c)$ and $L$ are left-invertible,
while  $\Con{N}(A,b)$ and $R$ are
right-invertible. Therefore,
\begin{align*}
\tilde A&=L^\dagger H_g(2,N) R^\dagger\\ 
&=S^{-1}\Obs{N}(A,c)^\dagger \Obs{N}(A,c) A\Con{N}(A,b)\Con{N}(A,b)^\dagger S \\ &=S^{-1}AS\;,
\end{align*}
i.e., the triple $(\tilde A,\tilde b,\tilde c)$ is similar to
$(A,b,c)$ and 
\begin{align*}\tilde{\mathcal{O}}(N)=\Obs{N}(A,c)S=L,\;
\tilde{\mathcal{C}}(N)=S^{-1}\Con{N}(A,b)=R.
\end{align*} If $L$ and $R$ are chosen from a singular value decomposition $H_g(1,N) =
U(N)\Sigma(N) V(N)^\transp$ as $L=U(N)\Sigma(N)^{\frac{1}{2}}$ and
$R=\Sigma(N)^{\frac{1}{2}}V(N)^\transp $, then $$\tilde Q=LL^\transp
=\Sigma(N)=R^\transp R=\tilde P,$$ i.e.\ the realization is
finite-time balanced. This approach is known as \emph{Kung's
  algorithm}, \cite{kung1987identification}, see also
\cite[p.~74]{markovsky2012low}. Note that $H_g(1,N)$ is symmetric and therefore $U(N)$ and $V(N)$ are
equal up to the column signs.

We denote by $(A_r(N),b_r(N),c_r(N))$ the truncation of $(\tilde A,\tilde b,\tilde c)$ to an $r$-th order
approximation. %
By the convergence of the Gramians \eqref{eq:convergence_of_Gramians}
it follows that  $(A_r(N),b_r(N),c_r(N))$ converges also for $N\to\infty$.
\begin{prop}\label{prop:kung_bt}
	For $G(z)$ and $N > n$, $(A_r(N),b_r(N),c_r(N))$ is a finite-time balanced truncated approximation of $G(z)$ and
        \begin{align*}
(A_r,b_r,c_r) &:= \lim_{N\to \infty}  (A_r(N),b_r(N),c_r(N))
        \end{align*}
is a balanced truncated approximation. 
\end{prop}

\section{$k$-positivity theory}
\label{sec:k_pos_theory}
Let us now introduce the framework of $k$-positivity, which has been studied extensively in the monograph \cite{karlin1968total}. We begin by a discussion of finite dimensional matrices and continue with recent results on Hankel operators, whose approximation is the main subject of this work.   
\subsection{$k$-positive matrices}
A remarkable feature of nonnegative matrices is the Perron-Frobenius theorem \cite{perron1907theorie,frobenius1912matrizen}.
\begin{prop}[Perron-Frobenius]\label{PF}
Let $A \in \Rnn_{\geq 0}$. %
\begin{enumerate}
\item $\lambda_1(A) \geq 0$.
\item If $\lambda_1(A)$ has algebraic
  multiplicity $m_0$, then $A$ has $m_0$ linearly independent
  nonnegative eigenvectors related to $\lambda_1(A)$.
\item If $A$ is irreducible, then $m_0=1$, $\lambda_1(A)>0$ and $A$ has a
  strictly positive eigenvector related to $\lambda_1(A)$.
\end{enumerate}
\end{prop}
Obviously, all 1-minors of a nonnegative matrix $A$ are nonnegative. %
A generalization of this property is provided through the concept of 
\emph{multi-positivity}, which is central in our further approach. To introduce it, we need the notion of an $r$-th compound matrix. Consider the set of sorted $r$-tuples of $\{1,\ldots,n\}$ given by 
\begin{equation*}
\mathcal{I}_{n,r} := \{ v = \{v_1,\dots,v_r\} : 1\leq v_1 < v_2 < \dots < v_r \leq n \},
\end{equation*} 
where $\mathcal{I}_{n,r}$ is \emph{ordered lexicographically}. 
The $(i,j)$-th entry of the \emph{r-th multiplicative compound} matrix $\compound{X}{r} \in \mathbb{R}^{\binom{n}{r} \times \binom{m}{r}}$ to $X \in \Rmn$ is then defined by $\det(X_{(I,J)})$, where $I$ is the $i$-th and $J$ is the $j$-th element in $\mathcal{I}_{n,r}$ and $\mathcal{I}_{m,r}$, respectively. For example, if $X \in \mathbb{R}^{3 \times 3}$, then $\compound{X}{r}$ reads
\begin{align*}
\begin{pmatrix}
\det(X_{\{1,2 \},\{1,2 \}}) & \det(X_{\{1,2 \},\{1,3\}}) & \det(X_{\{1,2 \},\{2,3\}})\\
\det(X_{\{1,3 \},\{1,2 \}}) & \det(X_{\{1,3 \},\{1,3\}}) & \det(X_{\{1,3 \},\{2,3\}})\\
\det(X_{\{2,3 \},\{1,2 \}}) & \det(X_{\{2,3 \},\{1,3\}}) & \det(X_{\{2,3 \},\{2,3\}})\\
\end{pmatrix}.
\end{align*}
By the Cauchy-Binet formula \cite{horn2012matrix}, one can show the
following properties (see e.g.\ \cite[Chapter 6]{fiedler2008special}).
\begin{lem}\label{lem:compound_mat}
	Let $X \in \mathbb{R}^{n \times p}$, $Y \in \mathbb{R}^{p \times m}$ and $r \in \mathds{Z}_{\geq 1}$.
	\begin{enumerate}[i)]
		\item $\compound{(XY)}{r} = \compound{X}{r}\compound{Y}{r}$.
		\item If $p = n$, then $\sigma(\compound{X}{r}) = \{\prod_{i
                    \in I} \lambda_i(X): I \in \mathcal{I}_{n,r} \}$.
                  Moreover, if for $i\in I$ the columns $v_i$ of $V_I\in \mathbb{C}^{n\times r}$ are eigenvectors of $X$
                  corresponding to $\lambda_i$, then $C_r(V_I)$
              is an eigenvector of
                  $\compound{X}{r}$   corresponding to %
$\prod_{i
                    \in I} \lambda_i(X)$. 
		\item $\compound{(X^\transp)}{r} = \compound{X}{r}^\transp$.
		\item If $X \succeq 0$, then $\compound{X}{r} \succeq 0$.
	\end{enumerate} 
\end{lem}
\begin{defn}
	Let $X \in \Rmn$ and $k \leq \min\{m,n\}$. Then, $X$ is called \emph{(strictly) $k$-positive} if all $j$-minors of $X$ are (positive) nonnegative for $1 \leq j \leq k$. If $k = \min\{m,n\}$, we call $X$ \emph{(strictly) totally positive}.
\end{defn}
By \cref{lem:compound_mat,PF}, it holds therefore for strictly
$k$-positive $X \in \Rnn$ that $X$ is a nonnegative matrix with
$\lambda_1(X) > \dots >\lambda_k(X) > 0$. This extends the result on the Perron-Frobenius eigenvalue $\lambda_1(X)$. In particular, we have the following important properties \cite{fallat2017total}. 
\begin{lem}\label{lem:Hankel_EW}
Let $(S)HP_k \subset \Rnn$ denote the set of all (strictly) $\tp{k}$ Hankel matrices. Then,
\begin{enumerate}[i.]
    \item $HP_k$ is a proper convex cone.
    \item $SHP_k$ lies densely in $HP_k$.
    \item If $X_1 \in HP_{k_1}$ and $X_2 \in HP_{k_2}$, then
    \begin{enumerate}
        \item $\lambda_1(X_1)\geq \dots \geq \lambda_{k_1}(X_1) \geq 0$.
        \item $X_1+X_2 \in HP_{\tp{\min\{k_1,k_2\}}}$
    \end{enumerate}
\end{enumerate}
	\end{lem}

\subsection{Hankel $k$-positive systems}
Next, we review LTI systems with $G(z)$ given by \eqref{eq:defGz}, whose Hankel operator representation matrix \eqref{eq:def_hank_disc} is $k$-positive. These systems are the main interest of this paper. The results stated here can be found in \cite{grussler2020variation}.
\begin{defn}[Hankel $k$-positivity] \label{def:hankel_kpos}
$G(z)$ is called \emph{Hankel (strictly) $\tp{k}$} if $H_g(1,N)$ is (strictly) $\tp{k}$ for all $N \geq k$. We say that $G(z)$ is \emph{Hankel (strictly) totally positive} if $k = \infty$. 
\end{defn}
In case of $k=1$, this means that $g \geq 0$. As such system map nonnegative inputs to nonnegative outputs, they are also called \emph{externally positive}.
An important sub-class of externally positive systems is formed through so-called internal positivity.
\begin{defn}[Internal positivity]
	$G(z)$ has an internally positive realization $(A,b,c)$ if $A$, $b$ and $c$ are nonnegative. 
\end{defn}
There exists several sufficient certificate for external positivity \cite{grussler2019tractable,drummond2019external}. Fortunately, also in case of $k > 1$, we do not need to check all minors of $H_g(1,N)$, but it suffices to verify external positivity of the so-called \emph{$j$-th compound system $G_{[j]}(z)$} with $g_{[j]}(t) := \det(H_g(t,j))$, $1 \leq j \leq k$.

\begin{prop}\label{prop:Hankel_minor}
	For $G(z)$ and $k \leq n$, the following are equivalent: 
	\begin{enumerate}
		\item $G(z)$ is Hankel $k$-positive.
		\item $G_{[j]} \geq 0$ is externally positive, $1 \leq j \leq k$.
		\item $H_g(1,k-1) \succ 0$, $H_g(2,k-1) \succeq 0$ and $G_{[k]}$ is externally positive. 
	\end{enumerate}
\end{prop}
	
Note that $G_{[j]}$ are of finite order as $G_{[j]}$ has the realization $(\compound{A}{j}, \compound{\Con{j}(A,b)}{j},\compound{\Obs{j}(A,c)}{j})$.

\begin{example}\label{ex:sum_first}
	The simplest example of a Hankel totally positive system is $G(z) = \sum_{i=1}^n \frac{r_i}{z-p_i}$ with $r_i, p_i \geq 0$. Indeed, for each system $(p_i,r_i,1)$, it holds for $j \geq 2$ that $\rk(\mathcal{O}(j)) = 1$, which is why $C_j(\mathcal{O}(j)) = 0$ and thus $g_{[j]} = 0$. First order externally positive systems are therefore Hankel totally positive and by \cref{lem:Hankel_EW} also their sums. 
\end{example} 
First order systems are indeed the prototypes of $k$-positivity.
\begin{prop}\label{prop:tpk_hank}
	Let $G(z) = \sum_{i=1}^n \frac{r_i}{z-p_i}$ have distinct poles and be Hankel $k$-positive with $n \geq k \geq2$. Then,
	\begin{equation}
	G(z) = \frac{r_1}{z-p_1} + G_r(z) \quad \text{where} \quad \Hank{g_r} \text{ is $k-1$-positive}
	\end{equation}
	with $r_1 > 0$ and $p_1 \geq 0$.
\end{prop}
{{In particular, a repeated application of \cref{prop:tpk_hank} implies that $G(z) = \sum_{i=1}^{k} \frac{r_i}{z-p_i} + G_r(z)$ with $r_i > 0$, $p_i \geq 0$ and $G_r(z)$ only containing poles of smaller magnitude. The dominant dynamics of $G(z)$ are, therefore, Hankel totally positive. For $k = n$, we have the following necessary and sufficient decomposition known from relaxation systems.}}
\begin{cor}\label{cor:totally_hank}
	$G(z)$ is Hankel totally positive if and only if $G(z) = \sum_{i=1}^{n} \frac{r_i}{z-p_i}$, where $r_i > 0$ and $p_i\geq0$. 
\end{cor}
In other words, Hankel $k$-positivity is a framework that quantifies the transit from external positivity -- one dominant first order lag -- to relaxation systems -- sums of first order lags.

\section{Reduction of $k$-positive Hankel operators}\label{sec:Main}
Next, we look into balanced truncation of $k$-positive Hankel
operators. We start with state-space symmetric
systems as an intermediate step. Then, we treat the totally positive
case, before we finally prove Theorem \ref{thm:main} as our  general main result. 
\subsection{Balanced truncation of state-space symmetric systems}
\begin{defn}
	A realization $(A,b,c)$ is called \emph{state-space symmetric} if $A = A^\transp$ and $b^\transp = c$.
\end{defn}
The following characterizations of state-space symmetric systems holds. 
\begin{prop}\label{prop:symm}
Let $G(z)$ be of order $n$. Then the following are equivalent:
\begin{enumerate}
	\item $G(z)$ has a state-space symmetric minimal realization. \label{item:sym_real}
	\item $H_g(1,n) \succ 0$.  \label{item:psd}  
	\item $G(z) = \sum_{i=1}^n \frac{r_{i}}{z-p_{i}}$ with $r_{i} > 0$ and $p_i \in \mathbb{R}$ for all $i$. \label{item:parallel}
	\item If $(A,b,c)$ is a minimal realization of $G(z)$ with cross-Gramian $X$, then $\sigma(X) \subset \mathds{R}_{> 0}$. \label{item:cross}
	\item $G(z)$ has a balanced state-space symmetric minimal realization. \label{item:sym_real_bal}
\end{enumerate}  
\end{prop}
\begin{proof}
  \ref{item:cross}) $\Rightarrow$ \ref{item:psd}) Since
  $X=\lim_{N\to\infty} {\cal C}(N){\cal O}(N)$ and $\sigma(X)\subset
  \mathds{R}_{> 0}$, there is an $N>n$, such that $$\mathds{R}_{>
    0}\supset\sigma({\cal C}(N){\cal O}(N))=\sigma({\cal O}(N){\cal
    C}(N))\setminus\{0\}.$$ Hence, ${\cal O}(N){\cal C}(N)=H_g(1,N) \succeq 0$ and as such its principle sub-matrix $H_g(1,n) \succeq 0$. Since $H_g(1,n)$ is non-singular, $H_g(1,n)\succ 0$.\\
\ref{item:psd}) $\Rightarrow$ \ref{item:sym_real_bal}) If
$H_g(1,n)\succ 0$ then it has a symmetric SVD $H_g(1,n)=U\Sigma U^\transp$
and the balanced realization obtained by Kung's algorithm is
symmetric.\\
\ref{item:sym_real_bal}) $\Rightarrow$ \ref{item:parallel}) 
By symmetry of the realization we have $G(z)=b^\transp (zI-A)^{-1} b$.
If $S^\transp AS=\diag(p_1,\dots,p_n)\subset\mathbb{R}^{n\times n}$ is the spectral decomposition of $A$
and $S^\transp b=\tilde b$, then $G(z)=b^\transp S(zI-S^\transp A
S)^{-1}S^\transp b=\sum_{i=1}^n \frac{\tilde b_{i}^2}{z-p_{i}}$.\\
\ref{item:parallel}) $\Rightarrow$ \ref{item:sym_real}) 
If $c=[\sqrt{r_1},\ldots,\sqrt{r_n}]$, $b=c^\transp$, and
$A=\diag(p_1,\dots,p_n)$, then $G(z)=c(zI-A)^{-1}b$. Hence we have a
symmetric minimal realization.\\
\ref{item:sym_real}) $\Rightarrow$ \ref{item:cross}) If $A=A^\transp$
and $b=c^\transp$, then all Gramians are equal, $P=Q=X$.  In
particular $X\succ 0$, if the realization is minimal.
\end{proof}

The last item in \cref{prop:symm} yields the following property of balanced truncation. 
\begin{cor}\label{cor:symm_bt}
	Balanced truncation preserves state-space symmetry, i.e., all truncated models are state-space symmetric. 
\end{cor}
In fact, this property is also shared by optimal Hankel-norm approximation \cite{liu1998model}.

\subsection{Balanced truncation of totally positive Hankel operator}
A comparison with \cref{cor:totally_hank} reveals that state-space symmetric systems fulfil many of the requirements necessary for Hankel total positivity. {However, there is an important difference, which manifests itself as follows.}
\begin{cor}\label{cor:total_pos}
	For $G(z)$, the following are equivalent:
	\begin{enumerate}
		\item $G(z)$ is Hankel totally positive. \label{item:tot_pos}
		\item $H_g(1,n)\succ 0$ and $H_g(2,n) \succeq 0$. \label{item:psd_2}
		\item $G(z) = \sum_{i=1}^n \frac{r_i}{z-p_i}$ with $r_i > 0$ and $p_i \geq 0$ for all $i$. \label{item:parallel_pos}
		\item $G(z)$ has an internally positive state-space symmetric realization. \label{item:int_pos}
		\item $G(z)$ has a balanced minimal state-space symmetric realization $(A,b,c)$ with $A \succeq 0$.
		\label{item:bal_real_pos}
	\end{enumerate}
\end{cor} 
\begin{proof}
  \ref{item:tot_pos}) $\Rightarrow$  \ref{item:psd_2})
By definition, total positivity 
implies $H_g(1,n)\succeq 0$ and $H_g(2,n) \succeq 0$. Since $G$ has
order $n$, it follows that $H_g(1,n)$ is nonsingular.\\
\ref{item:psd_2}) $\Rightarrow$ \ref{item:bal_real_pos} Since
$H_g(1,n)\succ 0$, we can factorize $H_g(1,n)=LL^\transp$ to obtain a
balanced symmetric minimal realization, where $b=c^\transp$ is the first
column of $L$ and $A=L^\dagger H_g(2,n) (L^\dagger)^\transp\succeq 0$.\\
\ref{item:bal_real_pos} $\Rightarrow$  \ref{item:int_pos})
This is obvious.\\
\ref{item:int_pos}) $\Rightarrow$  \ref{item:parallel_pos}) As in
\cref{item:parallel} of Proposition \ref{prop:symm}, we obtain the
partial fraction expansion of $G$ where now additionally $p_i\ge 0$,
since $A\succeq 0$.\\
 \ref{item:bal_real_pos} $\Rightarrow$ \ref{item:tot_pos}
This has been discussed in \cref{ex:sum_first}.
\end{proof}

The equivalence of the first two items has already been noted in \cite[Theorem~4.4]{pinkus2009totally}, but since we use realization theory, its proof is greatly simplified and also provides an alternative proof of \cref{cor:totally_hank}. The last item in \cref{cor:total_pos} implies the following property of balanced truncation \cite{liu1998model}. 

\begin{prop}
		Let $G(z)$ be Hankel totally positive. Then, balanced truncation yields Hankel totally positive approximations.
\end{prop}

\subsection{Balanced truncation of Hankel $k$-positive systems}
{While the previous results have well-known analogues for continuous-time systems \cite{grussler2012symmetry,liu1998model,willems1976realization}, the general case is our main result, which follows from the following lemma.}
\begin{lem}\label{lem:poles_Hankel}
	Let $G(z)$ be Hankel $k$-positive with $k \leq n$. If $(A,b,c)$ is a minimal realization of $G(z)$ with cross-Gramian $X$, then $\lambda_1(X),\dots,\lambda_{k}(X) > 0$. 
\end{lem}
\begin{proof}
 	Using \cref{lem:Hankel_EW}, it follows for $N \geq k$ that $\lambda_1(H_g(1,N)),\dots,\lambda_{k}(H_g(1,N)) > 0$. Since $\lambda_i(\Hank{g})  = \lim_{N\to\infty} \lambda_{i}(H_g(1,N))$, by the continuity of the eigenvalues (see e.g. \cite{horn2012matrix}), the result follows because $\rk(\Hank{g}) = n$ and $\lambda_i(X) = \lambda_i(\Hank{g})$.
\end{proof}

\begin{thm}\label{thm:main}
	Let $G(z)$ be Hankel $k$-positive and $r \leq k$. Then, if {$\sigma_r(\Hank{g}) \neq \sigma_{r+1}(\Hank{g})$}, balanced truncation to order $r$ yields an asymptotically stable Hankel totally positive approximation.
\end{thm}
\begin{proof}
Since Hankel $k$-positivity implies Hankel $r$-positivity, $r \leq k$, it suffices to consider the case $k = r$. It is known that
balanced truncation to order $k$ preserves asymptotic stability, if
$\sigma_k(\Hank{g})>\sigma_{k+1}(\Hank{g})$ (e.g.\
\cite{hinrichsen1990improved}). \\
To prove total positivity assume first that $G$ is \emph{strictly} Hankel $k$-positive.
  As before let $\sigma_i(H_g(1,N))$ denote the $i$-th singular value of
  $H_g(1,N)$. Then $\sigma_i(H_g(1,N))$ converges to
  $\sigma_i(\Hank{g})$ for $N\to\infty$. Hence, for sufficiently large
  $N$, we have $\sigma_k(H_g(1,N))>\sigma_{k+1}(H_g(1,N))$. Since
  $G(z)$ is Hankel $k$-positive, all $H_g(1,N)$ are $k$-positive and
  thus  $\sigma_i(H_g(1,N))=\lambda_i(H_g(1,N))$ for $i=1,\ldots,k$.
Let $u_1(N),\ldots,u_k(N)$ be a corresponding set of orthonormal
eigenvectors and define $U_j(N)=[u_1(N),\ldots,u_j(N)]\in\mathbb{R}^{N\times
  j}$ for $1\le j\le k$.  Then, following
  Kung's algorithm described in subsection \ref{sec:Kung}, a balanced
  truncated approximation is given by
  $$A_k(N)=\Sigma_k(N)^{-\frac12}U_k(N)^\transp H_g(2,N)
  U_k(N)\Sigma_k(N)^{-\frac12},$$ $c_k(N)=b_k(N)^\transp$ equal to the 1st row
  of $\Sigma_k(N)^{-\frac12}U_k(N)$. \\
It is evident, that $A_k(N)$ is symmetric.
In view of Corollary \ref{cor:total_pos}, we need to show that
$A_k(N)\succeq0$.  This follows from Sylvester's criterion, if
\begin{align}\label{eq:Sylvester_condition}
\det \left(U_j(N)^\transp H_g(2,N) U_j(N)\right)&>0
\end{align}
 for all
$j=1,\ldots, k$.  By Lemma \ref{lem:compound_mat} the compound matrix
$C_j(U_j(N))$ is an eigenvector of the positive matrix $C_j(H_g(1,N))$ corresponding to
the eigenvalue $$\lambda_1(C_j(H_g(1,N)))=\prod_{i=1}^j
\lambda_i(H_g(1,N))>0.$$ Hence we can assume that $C_j(U_j(N))$ is
positive (see also Remark \ref{rem:assumeCjpositive} below).   Together with the positivity of $H_g(2,N)$ we have
\begin{align*}
0&<C_j(U_j(N))^T C_j(H_g(2,N))C_j(U_j(N))\\&= C_j(U_j(N)^T H_g(2,N) U_j(N))\\&=\det(U_j(N)^T H_g(2,N)U_j(N)), 
\end{align*}
which is \eqref{eq:Sylvester_condition}. \\
We conclude that the $N$-balanced reduced system is strictly totally
positive. By \cref{lem:Hankel_EW}, the result follows also for the
non-strict case. Finally, letting $N\to\infty $ yields the
corresponding statements for $\Hank{g}$.
\end{proof}
Thus, systems with $k$-positive Hankel operators have approximations
that naturally correspond to their characteristic dominant dynamics. In particular, we want to single out the following important case for $k=1$. 
\begin{cor}
	Let $G(z)$ be externally positive. Then, its first order balanced truncated approximation is externally positive.  
\end{cor}
\begin{rem}\label{rem:assumeCjpositive}
\begin{enumerate}
    \item  A word on the assumption $C_j(U_j(N))>0$ in the previous
    proof might be helpful. By Lemma \ref{lem:compound_mat} there
    exist eigenvectors $\tilde u_1(N),\ldots,\tilde u_k(N)$, forming a
    matrix $\tilde U_j(N)$, such that
    $C_j(\tilde U_j(N))>0$ for all $j\le k$. These eigenvectors may
    differ from $u_1(N),\ldots,u_k(N)$, but span the same space.
    Therefore $U_k(N)=\tilde U_k(N)S$ where $S$ is an orthogonal
    matrix. This transformation amounts to a similarity transformation
    of the reduced system.
\item If we drop the assumption that
  $\sigma_r(\Hank{g})>\sigma_{r+1}(\Hank{g})$ then the reduced system might
  not be asymptotically stable. Moreover, our proof does not guarantee
  total positivity of every balanced truncated approximation to order
  $k$, although it still holds true that there exists such a truncation.
\end{enumerate}
\end{rem}

\section{Multi-Input-Multi-Output Systems}
It is easy to see that our results extend to MIMO systems with symmetric Hankel operators, i.e., $\Hank{g} = \Hank{g}^\transp$. However, the following result for internally positive systems suggests that we can even leap beyond that.
\begin{thm}
	Let $(A,B,C)$ be an internally positive MIMO system. Then, there exists an asymptotically stable, internally positive, balanced truncated first order approximation. \label{thm:first_order_internal}
\end{thm}
\begin{proof} 
	Let $P$ and $Q$ be the controllability and observability Gramians of $(A,B,C)$. Obviously, $P, Q \in \Rnn_{\geq 0}$ and thus $PQ \in \Rnn_{\geq 0}$, too. Balancing
	the system via a state-space transformation $x = T\xi$ yields
	$T^{-1}PQT = \text{diag}\begin{pmatrix} \Sigma^2,0 \end{pmatrix}$,
	where $\Sigma = \text{diag}\begin{pmatrix} \sigma_1 I_{k_1}, \dots ,
	\sigma_N I_{k_N}\end{pmatrix},$ containing the Hankel singular
	values $\sigma_1 > \cdots > \sigma_N$, with corresponding
	multiplicities $k_1,\dots,k_N$. Hence, the columns of $T$ are eigenvectors of $PQ$ and by
	\cref{PF} there exists a nonnegative right-eigenvector $v_1$ to the largest eigenvalue $\sigma_1$, i.e. $PQ v_1 = \sigma_1 v_1 \ \text{with} \ T = \begin{pmatrix}v_1,\dots,v_n
	\end{pmatrix}$. Analogously, there is a nonnegative left-eigenvector $w_1$ with $T^{-1} = \begin{pmatrix}
	w_1,\dots ,w_n \end{pmatrix}^T$. If $k_1 = 1$, the asymptotic stability of the reduced system of order 1 is given by  nonnegative $B_1 = w_1^{T}B$ and $C_1 = C v_1 \geq 0$ as well as $A_1 = w_1^{T}Av_1$, where $A_1$ is  positive in discrete-time and negative in continuous-time.
	
	If $k_1>1$, it could happen that $A_1$ is only marginally stable. But since the reduced
	system of order $k_1$ (belonging to all $\sigma_1$) is asymptotically stable, there
	must exist at least one asymptotically stable first order
	approximation. Further, by \cref{PF} we conclude the reducibility of $PQ$ and thus the internal positivity of each first order approximation. %
\end{proof}

\section{Example}
\label{sec:example}
We consider an illustrative example to demonstrate how Hankel $k$-positivity emerges from relaxation systems as well as to show that Hankel $k$-positive system do not allow for much larger Hankel totally positive approximations than up to order $k$. To this end, let
\begin{align*}
G_k(z) = \sum_{j=1}^6 \frac{1}{z-\frac{1}{10-j}} - \frac{r_k}{z-0.3},  r_k \geq 0
\end{align*}
where the parameter vector $r=(r_1,r_2,r_3,r_4,r_5,r_6,r_7)$,
$$r = \begin{pmatrix}
6 &  1.1538 & 0.3125 & 0.0769 & 0.0132 & 0.0011 & 0
\end{pmatrix}$$ contains the threshold values up to which $G_k(z)$ is Hankel 
$k$-positive. Note, e.g., that by \cref{cor:totally_hank}, $G_k(z)$ cannot be Hankel totally positive for $r_k > 0$. 
For each $r_k$, the largest orders $o_k$ for which balanced truncation of $G_k(z)$
yields a relaxation system are then contained  in the vector $o = \begin{pmatrix}1 & 2 & 4 & 5 & 6 & 6 & 6
\end{pmatrix}$. This demonstrates that the positivity degree may be quite sharp for determining a priori the largest truncation order for which Hankel totally positive approximations can be expected. 

{{It follows from \cref{cor:totally_hank} that $o_k$ also determines the order up to which balanced truncation returns an internal positive realizable approximation, which is independent of a particular system realization. In contrast, \cite{reis2009positivity,sootla2012scalable,feng2010internal} require internally positive realizations to begin with, which leads to internally positive approximations with conservative errors after the reduction of only a few states \cite{grussler2012model,grussler2012symmetry}. For example, applying \cite{reis2009positivity} for obtaining a fifth order approximation of $G_7(z)$ with realization $A = \diag(0.9,\dots,0.4)$, $b^\transp = c = \begin{pmatrix} 1 & \dots & 1\end{pmatrix}$, simply removes the dynamics of the fastest pole, resulting in a relative $H_{\infty}$-error of $6.8 \cdot 10^{-2}$. Balanced truncation to order $2$, however, has only an error of $8.8 \cdot 10^{-3}$. 

Our example, further, suggests that small imperfection, e.g., in the measurement of the impulse response may make it impossible to identify a truly underlying Hankel totally positive system. Then, our results indicate that balanced truncation may be used to damp the contribution of these imperfections by finding a nearby Hankel totally positive approximation. 

Finally, note that systems such as $G_k(z)$ can be found as the linear part of a perceptron within neural networks \cite{grussler2020variation}. }}

	\section{Conclusion}
	In this work, we have addressed the problem of finding reduced order models that consist of a parallel interconnection of first order lags. While approximating a system with a relaxation or an internally positive system generally requires new algorithms, our results show that for the class of Hankel $k$-positive systems it suffices to use balanced truncation. Interestingly, this proves that balanced truncation yields approximations, which are of the same form as the system's dominant dynamics. So far, this has only been observed for the reduction of relaxation systems \cite{liu1998model}. In particular, reduction of an externally positive system to order $1$ will always provide an internally positive approximation, which often outperforms specialized internally positivity preserving reduction methods. Further, our example indicates that the Hankel positivity degree is often close to the largest possible order for which balanced truncation yields a relaxation system. 
	
	Nonetheless, our results also face limitations:
        (i) for large system, it may be computationally difficult to verify Hankel
        $k$-positivity, (ii) our results mainly apply to systems with symmetric Hankel operator. In the future, we hope to overcome the first
        limitation through extensions to the class of Hankel internally
        $k$-positive systems, i.e., systems with internally positive compound
        systems. In particular, as this class requires $A$ to be
        $k$-positive, it will connect to recent
        investigations of autonomous internally $k$-positive systems
        in \cite{weiss2019generalization,margaliot2018revisiting,seidi2020discrete,weiss2021cooperative}. Concerning the second limitation, our result on the reduction of internal positive systems to order~$1$ indicates that extensions to systems with non-symmetric Hankel operator are plausible. 
        
      {{In the future, it would be interesting to extend these results to the Toeplitz operator. Another important question is whether our results also extend to optimal low-rank Hankel approximations. The example in \cite{grussler2016low} suggests an affirmative answer. In particular, this would result in so-called \emph{completely positive} approximations \cite{berman1979nonnegative} with the attractive feature of having a rank-revealing nonnegative matrix factorization \cite{gillis2020nonnegative}.    
        
        Finally note that our results can also be readily extended to continuous-time systems. }}

\section*{Acknowledgment}
The research leading to these results was completed while the first author was a postdoctoral research associate at the University of Cambridge. The research has received funding from the European Research Council under the Advanced ERC Grant Agreement Switchlet n.670645.

\bibliographystyle{plain}

\bibliography{refkpos,refopt,refpos,science}

\end{document}